\documentclass[a4,12pt]{amsart}
\usepackage{amssymb,amstext,amsmath,amscd,amsthm,amsfonts,enumerate,graphicx,latexsym}
\usepackage[usenames]{color}
\newtheorem{thm}{Theorem}[section]
\newtheorem{cor}[thm]{Corollary}
\newtheorem{lem}[thm]{Lemma}
\newtheorem{prop}[thm]{Proposition}
\newtheorem*{thm*}{Theorem}
\newtheorem*{cor*}{Corollary}
\newtheorem*{thma}{Theorem \ref{1.2}}
\newtheorem*{thmb}{Theorem \ref{1.3}}
\newtheorem*{thmc}{Theorem \ref{hyper}}
\newtheorem*{thmd}{Theorem \ref{1.4}}
\theoremstyle{definition}
\newtheorem{dfn}[thm]{Definition}
\newtheorem{rem}[thm]{Remark}

\newtheorem*{conv}{Convention}
\newtheorem{ex}[thm]{Example}

\newtheorem*{claim*}{Claim}

\newtheorem{nota}[thm]{Notation}
\theoremstyle{remark}

\renewcommand{\qedsymbol}{$\blacksquare$}
\numberwithin{equation}{thm}
\def\X{\mathcal{X}}
\def\Y{\mathcal{Y}}
\def\ZZ{\mathcal{Z}}
\def\OO{\mathcal{O}}

\def\mod{\operatorname{\mathsf{mod}}}
\def\Spec{\operatorname{\mathsf{Spec}}}
\def\add{\operatorname{\mathsf{add}}}
\def\Max{\operatorname{\mathsf{Max}}}
\def\NF{\operatorname{\mathsf{NF}}}
\def\FPD{\operatorname{\mathsf{PD}}}
\def\pd{\operatorname{\mathsf{pd}}}
\def\res{\operatorname{\mathsf{res}}}
\def\CM{\operatorname{\mathsf{MCM}}}

\def\depth{\operatorname{\mathsf{depth}}}
\def\Ext{\operatorname{\mathsf{Ext}}}
\def\height{\operatorname{\mathsf{ht}}}
\def\Ass{\operatorname{\mathsf{Ass}}}
\def\Im{\operatorname{\mathsf{Im}}}
\def\ext{\operatorname{\mathsf{ext}}}
\def\PD{\operatorname{\mathsf{PD}}}
\def\Sing{\operatorname{\mathsf{Sing}}}
\def\CMdim{\operatorname{\mathsf{CMdim}}}
\def\h{\operatorname{\mathsf{h}}}
\def\FD{\operatorname{\mathsf{L}}}

\def\IPD{\operatorname{\mathsf{IPD}}}
\def\Rfd{\operatorname{\mathsf{Rfd}}}
\def\grade{\operatorname{\mathsf{grade}}}
\def\Tor{\operatorname{\mathsf{Tor}}}
\def\Supp{\operatorname{\mathsf{Supp}}}
\def\Hom{\operatorname{\mathsf{Hom}}}
\def\Proj{\operatorname{\mathsf{Proj}}}

\def\N{\mathbb{N}}
\def\P{\mathbb{P}}

\def\syz{\mathsf{\Omega}}
\def\tr{\mathsf{Tr}}

\def\p{\mathfrak{p}}
\def\q{\mathfrak{q}}

\def\m{\mathfrak{m}}

\def\xx{\text{\boldmath $x$}}

\begin{document}
\setlength{\baselineskip}{15pt}
\title[Resolving subcategories and grade consistent functions]{Classification of resolving subcategories and grade consistent functions}
\author{Hailong Dao}
\address{Department of Mathematics, University of Kansas, Lawrence, KS 66045-7523, USA}
\email{hdao@math.ku.edu}
\urladdr{http://www.math.ku.edu/~hdao/}
\author{Ryo Takahashi}
\address{Graduate School of Mathematics, Nagoya University, Furocho, Chikusaku, Nagoya 464-8602, Japan/Department of Mathematics, University of Nebraska, Lincoln, NE 68588-0130, USA}
\email{takahashi@math.nagoya-u.ac.jp}
\urladdr{http://www.math.nagoya-u.ac.jp/~takahashi/}
\thanks{The first author was partially supported by NSF grants DMS 0834050 and DMS 1104017. The second author was partially supported by JSPS Grant-in-Aid for Young Scientists (B) 22740008 and by JSPS Postdoctoral Fellowships for Research Abroad}
\begin{abstract}
We classify certain resolving subcategories of finitely generated modules over a commutative noetherian ring $R$ by using integer-valued functions on $\Spec R$.
As an application, we give a complete classification of resolving subcategories when $R$ is a locally hypersurface ring or a  complete intersection.
Our results also recover and categorify a ``missing theorem'' by Auslander.
\end{abstract}
\maketitle

\section{Introduction}\label{intro}
Throughout the present paper, let $R$ be a commutative noetherian ring with identity which is not necessarily of finite Krull dimension.
All modules are assumed to be finitely generated.
We denote by $\mod R$ the category of (finitely generated) $R$-modules.
A resolving subcategory of $\mod R$ is a full subcategory that contains all projective modules and is closed under direct summands, extensions and syzygies.
(Here, a syzygy of an $R$-module $M$ means the kernel of some epimorphism from a projective $R$-module to $M$.)
In this paper we give a classification of certain resolving subcategories of $\mod R$ using functions from $\Spec R$ to the integers.
Our results can be applied to obtain classification of all resolving subcategories when $R$ is a locally hypersurface ring or a complete intersection.

There has been a flurry of research activities on classification of subcategories associated to algebraic objects in recent years.
These results are becoming increasingly influential in ring theory, homotopy theory, algebraic geometry and representation theory.
For instance, Gabriel \cite{Ga} classified all Serre subcategories of $\mod R$ using specialization closed subsets of $\Spec R$.
Devinatz, Hopkins and Smith \cite{DHS,HS} classified thick subcategories in the stable homotopy category.
Hopkins {\cite{Hop}} and Neeman {\cite{N}} gave a similar result for thick subcategories of the derived categories of perfect complexes over $R$.
The Hopkins-Neeman theorem was extended to schemes by Thomason \cite{T} and recently to group algebras by Benson, Iyengar and Krause {\cite{BIK}}.
For many other similar results, see \cite{BCR, K, S, wide, stcm} and the references therein.

Generally speaking, the less restrictive conditions one imposes on the subcategories the harder it is to classify them.
A well-known difficulty is  that $\mod R$ is ``too big", even when $R$ is nice (say, regular local).
Thus attentions are often restricted to more special categories such as that of maximal Cohen-Macaulay modules, or derived categories.
However, it has recently emerged that resolving subcategories of $\mod R$ are reasonable objects to look at, and they capture substantial information on the singularities of $R$, see \cite{radius,dim}.
Therefore, it is natural to ask if one can classify all such subcategories.

Our main results indicate that such  task is quite subtle but at the same time not hopeless.  The key new idea is to use certain integer-valued functions {on} $\Spec R$. Let us describe this class of functions which will be crucial for the rest of the paper.

\begin{dfn}\label{gcdef}
Let $\N$ denote the set of nonnegative integers.
Let $f$ be an $\N$-valued function on $\Spec R$.
We call $f$ {\it grade consistent} if it satisfies the following two conditions.
\begin{itemize}
\item
For all $\p\in\Spec R$ one has $f(\p)\le\grade\p$.
\item
For all $\p,\q\in\Spec R$ with $\p\subseteq\q$ one has $f(\p)\le f(\q)$.
\end{itemize}
\end{dfn}

\noindent
Here, recall that for a proper ideal $I$ of $R$, $\grade I$ denotes the grade of $I$, i.e., the minimum of the integers $n\ge0$ with $\Ext_R^n(R/I,R)\ne0$.
This is known to be equal to the common length of the maximal $R$-regular sequences in $I$.
We denote by $\FPD(R)$ the subcategory of $\mod R$ consisting of modules of finite projective dimension (we use the convention that the projective dimension of the zero module is $-\infty$, which is also regarded as finite).
One of our main results states

\begin{thm}\label{1.2}
Let $R$ be a commutative noetherian ring.
There is a 1-1 correspondence
$$
\begin{CD}
\left\{
\begin{matrix}
\text{Resolving subcategories of $\mod R$}\\
\text{contained in $\FPD(R)$}
\end{matrix}
\right\}
\begin{matrix}
@>{\phi}>>\\
@<<{\psi}<
\end{matrix}
\left\{
\begin{matrix}
\text{Grade consistent functions}\\
\text{on $\Spec R$}
\end{matrix}
\right\}.
\end{CD}
$$
Here $\phi,\psi$ are defined by $\phi(\X)(\p)=\max_{X\in\X}\{\pd X_\p\}$ and $\psi(f)=\{\,M\in\mod R\mid\text{$\pd M_\p\le f(\p)$ for all $\p\in\Spec R$}\,\}$.
\end{thm}

\begin{rem}
Only a few days after the first version of our preprint appeared on the arXiv, another new preprint \cite{HPST} was posted, which announces a classification of resolving subcategories in $\FPD(R)$ by solving the equivalent problem of classifying tilting classes over $R$.
That classification is essentially the same as Theorem \ref{1.2}, since one can get a descending sequence of specialization closed subsets as in \cite[Definition 3.1]{HPST} by taking $Y_i = \{\p \in \Spec R \mid f(\p) \geq i\}$.
Also in \cite{AJS}, a classification was obtained for compactly generated $t$-structures in the  derived categories of $R$. These structures can be seen to contain the resolving subcategories of $\FPD(R)$. However, it would take some work to derive our classification (or the one in \cite{HPST}) from this classification. In any case, our proofs are quite different and they give rather concrete information on when and how one can build a module from another with syzygies, extensions and direct summands, see Remark \ref{concrete} and Section \ref{exa}. 
\end{rem}

For $M\in\mod R$ we denote by $\add M$ the subcategory of $\mod R$ consisting of modules isomorphic to direct summands of finite direct sums of copies of $M$.
We say that a resolving subcategory $\X$ of $\mod R$ is {\em dominant} if for every $\p\in\Spec R$ there exists $n\ge0$ such that $\syz^n\kappa(\p)\in\add\X_\p$, where $\X_\p$ denotes the subcategory of $\mod R_\p$ consisting of modules of the form $X_\p$ with $X\in\X$.
Over a Cohen-Macaulay ring, the grade consistent functions also classify the dominant resolving subcategories.

\begin{thm}\label{1.3}
Let $R$ be a Cohen-Macaulay ring.
Then one has a 1-1 correspondence
$$
\begin{CD}
\left\{
\begin{matrix}
\text{Dominant resolving subcategories}\\
\text{of $\mod R$}\\
\end{matrix}
\right\}
\begin{matrix}
@>{\phi}>>\\
@<<{\psi}<
\end{matrix}
\left\{
\begin{matrix}
\text{Grade consistent functions}\\
\text{on $\Spec R$}
\end{matrix}
\right\},
\end{CD}
$$
where $\phi,\psi$ are defined by $\phi(\X)(\p)=\height\p-\min_{X\in\X}\{\depth X_\p\}$ and $\psi(f)=\{\,M\in\mod R\mid\text{$\depth M_\p\ge\height\p-f(\p)$ for all $\p\in\Spec R$}\,\}$.
\end{thm}

For an $R$-module $M$ we denote by $\IPD(M)$ the {\em infinite projective dimension locus} of $M$, i.e., the set of prime ideals $\p$ with $\pd_{R_\p}M_\p=\infty$.
For a subcategory $\X$ of $\mod R$, set $\IPD(\X)=\bigcup_{X\in\X}\IPD(X)$.
Let $\Sing R$ denote the {\em singular locus} of $R$, that is, the set of prime ideals $\p$ of $R$ such that $R_\p$ is not a regular local ring.
Let $W$ be a subset of $\Sing R$.
We say that $W$ is {\em specialization closed} if $V(\p)\subseteq W$ for every $\p\in W$.
We denote by $\IPD^{-1}(W)$ the subcategory of $\mod R$ consisting of modules whose infinite projective dimension loci are contained in $W$.
Recall that $R$ is said to be {\em locally hypersurface} if for each $\p\in\Spec R$ the local ring $R_\p$ is a hypersurface.
The above results allow us to give complete classification of resolving subcategories when $R$ is a locally hypersurface ring. Also, by combining with Stevenson's recent classification of thick subcategories of the singularity category \cite{S}, one obtains a complete classification of resolving subcategories for complete intersections: 

\begin{thm}\label{hyper}
\begin{enumerate}[\rm(1)]
\item
Let $R$ be a locally hypersurface ring.
There is a 1-1 correspondence
$$
\begin{CD}
\left\{
\begin{matrix}
\text{Resolving subcategories}\\
\text{of $\mod R$}
\end{matrix}
\right\}
\begin{matrix}
@>{\Phi}>>\\
@<<{\Psi}<
\end{matrix}
\left\{
\begin{matrix}
\text{Specialization closed}\\
\text{subsets of $\Sing R$}
\end{matrix}
\right\}
\times
\left\{
\begin{matrix}
\text{Grade consistent}\\
\text{functions on $\Spec R$}
\end{matrix}
\right\}.
\end{CD}
$$
One defines $\Phi,\Psi$ by $\Phi(\X)=(\IPD(\X),\phi(\X))$ with $\phi(\X)(\p)=\height\p-\min_{X\in\X}\{\depth X_\p\}$ and $\Psi(W,f)=\{\,M\in\IPD^{-1}(W)\mid\text{$\depth M_\p\ge\height\p-f(\p)$ for all $\p\in\Spec R$}\,\}$.
\item
Let $R=S/(\xx)$ where $S$ is a regular  ring and $\xx=x_1,\dots,x_c$ is a regular sequence on $S$.
Set $X=\P_S^{c-1}=\Proj S[y_1,\cdots,y_c]$ and let $Y$ be the zero subscheme of $\sum_{i=1}^cx_iy_i\in\Gamma(X,\OO_X(1))$.
Then one has a 1-1 correspondence
$$
\begin{CD}
\left\{
\begin{matrix}
\text{Resolving subcategories}\\
\text{of $\mod R$}
\end{matrix}
\right\}
\begin{matrix}
@>>>\\
@<<<
\end{matrix}
\left\{
\begin{matrix}
\text{Specialization closed}\\
\text{subsets of $\Sing Y$}
\end{matrix}
\right\}
\times
\left\{
\begin{matrix}
\text{Grade consistent}\\
\text{functions on $\Spec R$}
\end{matrix}
\right\}.
\end{CD}
$$
\end{enumerate}
\end{thm}

\begin{rem}\label{concrete} Using the above Theorem, one can decide whether a finitely generated module $N$ can be ``built" from another module $M$ by looking at computable invariants. As a simple example, let $R$ be a local complete intersection and $M, N$ be modules of finite length. Then $N$ can be built from $M$ using syzygies, extensions and direct summands if and only if the support variety (see \cite{AI}) of $N$ is contained in that of $M$. Support varieties can be computed using a computer program like Macaulay 2. 

\end{rem}

For an $R$-module $M$ we denote by $\res M$ its {\em resolving closure}, i.e., the smallest resolving subcategory of $\mod R$ containing $M$.
As another application we recover and give a categorical version of a ``missing'' result by Auslander.
One says that $\Tor$-rigidity holds for $\PD(R)$ if for any module $M\in \FPD(R)$ and $N\in \mod R$, $\Tor_i^R(M,N)=0$ forces $\Tor_j^R(M,N)=0$ for $j>i$.
It is known that $\Tor$-rigidity holds for $\PD(R)$ when $R$ is regular or a quotient of an unramified regular local ring by a regular element; see \cite{Aus, rigid, Lich}.
The equivalence of (1) and (3) in the theorem below for unramified regular local rings was announced by Auslander in his 1962 ICM speech but never published (cf. \cite[Theorem 3]{Aus}):  

\begin{thm}\label{1.4}
Let $R$ be a commutative noetherian ring.
Suppose that $\Tor$-rigidity holds for $\PD(R_\p)$ for all $\p \in \Spec R$.
The following are equivalent for $M,N \in \FPD(R)$:
\begin{enumerate}[\rm(1)]
\item
$\pd M_{\p} \leq\max\{\pd N_{\p},0\}$ for all $\p \in \Spec R$.
\item
$M\in\res N$.
\item
$\Supp \Tor_i^R(M,X) \subseteq \Supp \Tor_i^R(N,X)$ for all $i>0$ and all $X\in \mod R$.
\end{enumerate}
\end{thm}

Since the proofs are somewhat technical, we provide a short summary of our approach to the main results stated above.
To prove Theorem \ref{1.2} we started by considering a special but crucial case, namely to classify resolving subcategories which consist of modules of finite projective dimension and that are locally free on the punctured spectrum of $R$.
This is achieved in Section \ref{resFPD} (Theorem \ref{1.1}).
Here we are actually able to describe these subcategories of $\FPD(R)$ as the smallest extension-closed subcategories containing the transposes of certain syzygies of the residue field.

In Section \ref{rsomofpd} we recall or establish various technical reductive tools, such as how to approximate a module in some resolving subcategory by a module with smaller nonfree locus and how resolving subcategories localize (Lemma \ref{2182108} and Proposition \ref{2182101}).
We also prove Theorem \ref{1.4} in this section.

Section \ref{drsoacr} deals with dominant subcategories to prepare for Theorem \ref{1.3}.
Here a key point, Proposition \ref{5'}, is that any dominant subcategory containing a module of depth $t$ also contains a $t$-th syzygy of the residue field.
To show the main result of the section, Theorem \ref{g}, we also reduce to the case of modules locally free on the punctured spectrum.

The ingredients were put together in Section \ref{mainproofs} to give proofs of Theorems \ref{1.2} and \ref{1.3}.
Several examples showing importance of our results are given in Section \ref{exa}, and in the final section, Section \ref{reghyper}, the previous results are applied to prove Theorem \ref{hyper}.
In Theorem \ref{lci} we also prove for a locally complete intersection ring $R$, a quite general result giving precise connections between resolving subcategories of $\mod R$ and ones contained in $\CM(R)$ and $\PD(R)$.
This ``glueing'' result shows that for such rings one only need to classify resolving categories in $\CM(R)$ and $\PD(R)$ to get a full classification for $\mod R$.

\begin{conv}
For an $R$-module $M$, we denote its $n$-th syzygy by $\syz^nM$, and its transpose by $\tr M$.
(Recall that if $P_1\xrightarrow{\partial}P_0\to M\to0$ is a projective presentation of $M$, then $\tr M$ is defined to be the cokernel of the $R$-dual map of $\partial$.)
Whenever $R$ is local, we define them by using a {\em minimal} free resolution of $M$, so that they are uniquely determined up to isomorphism.
The depth of the zero $R$-module is $\infty$ as a convention.
\end{conv}

\section{Resolving subcategories in $\FPD_0(R)$}\label{resFPD}

Throughout this section, let $(R,\m,k)$ be a local ring.
An {\em extension-closed} subcategory of $\mod R$ is defined as a subcategory of $\mod R$ which is closed under direct summands and extensions.
For an $R$-module $M$ we denote by $\ext M$ the {\em extension closure} of $M$, that is, the smallest extension-closed subcategory of $\mod R$ containing $M$.
We denote by $\FPD_0(R)$ the subcategory of $\PD(R)$ consisting of modules that are locally free on the {\em punctured spectrum} $\Spec R\setminus\{\m\}$.
For an integer $n\ge0$, we denote by $\PD_0^n(R)$ the subcategory of $\mod R$ consisting of modules which are of projective dimension at most $n$ and locally free on the punctured spectrum of $R$.
This section is devoted to showing the theorem below that is an important step in proving Theorem \ref{1.2}.

\begin{thm}\label{1.1}
Let $R$ be a commutative noetherian local ring of depth $t$ and with residue field $k$.
Then one has a filtration
$$
\add R=\PD_0^0(R)\subsetneq\PD_0^1(R)\subsetneq\cdots\subsetneq\PD_0^t(R)=\FPD_0(R)
$$
of resolving subcategories of $\mod R$, and these are all the resolving subcategories contained in $\FPD_0(R)$.
Moreover, when $t>0$, for each integer $1\le n\le t$ one has
$$
\PD_0^n(R)=\res(\tr\syz^{n-1}k)=\ext(R\oplus\textstyle\bigoplus_{i=0}^{n-1}\tr\syz^ik).
$$
\end{thm}

The key is a detailed inspection of syzygies and transposes.
We begin with stating several basic properties of syzygies and transposes; the following proposition will be used basically without reference.

\begin{prop}\label{1}
For an $R$-module $M$ the following hold.
\begin{enumerate}[\rm(1)]
\item
Let $0\to L\to M\to N\to0$ be an exact sequence of $R$-modules.
Then one has exact sequences:
\begin{align*}
& 0\to\syz N\to L\oplus R^{\oplus a}\to M\to0,\\
& 0\to\syz L\to\syz M\oplus R^{\oplus b}\to\syz N\to0,\\
& 0\to N^\ast\to M^\ast\to L^\ast\to\tr N\to\tr M\oplus R^{\oplus c}\to\tr L\to0
\end{align*}
for some $a,b,c\ge0$.
\item
There are isomorphisms $(\tr M)^\ast\cong\syz^2M$ and $M^\ast\cong\syz^2\tr M\oplus R^{\oplus n}$ for some $n\ge0$.
\item
The module $\tr M$ has no nonzero free summand.
\item
One has an isomorphism $M\cong\tr\tr M\oplus R^{\oplus n}$ for some $n\ge0$.
Hence $\tr\tr M$ is isomorphic to a maximal direct summand of $M$ without nonzero free summand.
\item
There is an exact sequence
$$
0 \to \Ext^1(M,R) \to \tr M \xrightarrow{f} (\syz^2M)^\ast \to \Ext^2(M,R) \to 0$$
of $R$-modules with $\Im f\cong\syz\tr\syz M$.
\end{enumerate}
\end{prop}

\begin{proof}
(1) The first sequence is a consequence of a pullback diagram made by $0\to L\to M\to N\to0$ and $0\to\syz N\to R^{\oplus a}\to N \to 0$.
The second and third sequences are obtaind by the horseshoe and snake lemmas from the original sequence.

(2) These isomorphisms are obtained easily by definition.

(3) This statement follows from (the proof of) \cite[Lemma 4.2]{arg}.

(4) The former assertion is straightforward.
The latter is by (3).

(5) Applying \cite[Proposition (2.6)(a)]{AB} to $\tr M$ and using the first isomorphism in (2), we get such an exact sequence.
It follows from \cite[Appendix]{AB} that the image of $f$ is isomorphic to $\syz\tr\syz M\oplus R^{\oplus n}$ for some $n\ge0$.
We have surjections $\tr M\to\syz\tr\syz M\oplus R^{\oplus n}\to R^{\oplus n}$, which shows that $\tr M$ has a free summand isomorphic to $R^{\oplus n}$.
By (3) we have $n=0$.
\end{proof}

\begin{lem}\label{2}
Let $0\to L\to M\to N\to 0$ be an exact sequence of $R$-modules.
Let $n\ge0$ be an integer such that $\Ext^n(L,R)=0$.
Then there is an exact sequence
$$
0 \to \tr\syz^nN \to \tr\syz^nM\oplus R^{\oplus m} \to \tr\syz^nL \to 0.
$$
\end{lem}

\begin{proof}
We have an exact sequence $0\to\syz^nL\to\syz^nM\oplus R^{\oplus l}\to\syz^nN\to0$.
The following three exact sequences are induced:
\begin{align*}
& (\syz^nM\oplus R^{\oplus l})^\ast\xrightarrow{\alpha}(\syz^nL)^\ast\xrightarrow{\beta}\tr\syz^nN\to\tr\syz^nM\oplus R^{\oplus m}\to\tr\syz^nL\to0,\\
& (\syz^nM\oplus R^{\oplus l})^\ast\xrightarrow{\alpha}(\syz^nL)^\ast\xrightarrow{\gamma}\Ext^{n+1}(N,R)\xrightarrow{\delta}\Ext^{n+1}(M,R),\\
& 0=\Ext^n(L,R)\to\Ext^{n+1}(N,R)\xrightarrow{\delta}\Ext^{n+1}(M,R).
\end{align*}
Hence we observe: $\delta$ is injective, $\gamma$ is zero, $\alpha$ is surjective, and $\beta$ is zero.
\end{proof}

\begin{prop}\label{7}
Assume $\depth R=t>0$.
Let $M$ be an $R$-module which is locally free on the punctured spectrum of $R$.
Then for each integer $0\le i<t$, there exists an exact sequence
$$
0 \to \tr\syz^{i+1}\tr\syz^{i+1}M \to \tr\syz^i\tr\syz^iM\oplus R^{\oplus n_i} \to \tr\syz^i\Ext^{i+1}(M,R) \to 0.
$$
\end{prop}

\begin{proof}
Applying Proposition \ref{1}(5) to $\syz^iM$, we have an exact sequence $0 \to \Ext^{i+1}(M,R) \to \tr\syz^iM \to \syz\tr\syz^{i+1}M \to 0$.
The assumption implies that $\Ext^{i+1}(M,R)$ has finite length.
Since $i<t$, we have $\Ext^i(\Ext^{i+1}(M,R),R)=0$.
The assertion now follows from Lemma \ref{2}.
\end{proof}

\begin{lem}\label{4}
Assume $\depth R=t>0$.
Let $L\ne0$ be an $R$-module of finite length, and let $0\le n<t$ be an integer.
Take a minimal free resolution $\cdots\xrightarrow{d_2}F_1\xrightarrow{d_1}F_0\to L\to0$ of $L$.
Then one has an exact sequence
$$
0 \to F_0^\ast \xrightarrow{d_1^\ast} \cdots \xrightarrow{d_{n-1}^\ast} F_{n-1}^\ast \xrightarrow{d_n^\ast} F_n^\ast \xrightarrow{d_{n+1}^\ast} F_{n+1}^\ast \to \tr\syz^nL \to 0,
$$
which gives a minimal free resolution of $\tr\syz^nL$.
Hence,
$$
\syz^i\tr\syz^nL\cong
\begin{cases}
\tr\syz^{n-i}L & (0\le i\le n),\\
F_0^\ast & (i=n+1),\\
0 & (i\ge n+2).
\end{cases}
$$
In particular, $\pd(\tr\syz^nL)=n+1$.
\end{lem}

\begin{proof}
There is an exact sequence $0\to\syz^nL\to F_{n-1}\xrightarrow{d_{n-1}}\cdots\xrightarrow{d_1}F_0\to L\to0$.
Since $n<t$, we have $\Ext^i(L,R)=0$ for any $i\le n$.
Hence $0 \to F_0^\ast \xrightarrow{d_1^\ast} \cdots \xrightarrow{d_{n-1}^\ast} F_{n-1}^\ast \to (\syz^nL)^\ast \to 0$ is exact.
Splicing this with $0\to(\syz^nL)^\ast\to F_n^\ast\xrightarrow{d_{n+1}^\ast}F_{n+1}^\ast\to\tr\syz^nL\to0$, we are done.
\end{proof}

\begin{prop}\label{8}
Assume $\depth R=t>0$.
Let $L\ne0$ be an $R$-module of finite length.
Then $\res(\tr\syz^nL)=\res(\tr\syz^nk)$ for every $0\le n<t$.
\end{prop}

\begin{proof}
As $L\ne0$, there is an exact sequence $0 \to L' \to L \to k \to 0$.
Since $L'$ has finite length and $n<t$, we have $\Ext^n(L',R)=0$, and get an exact sequence $0 \to \tr\syz^nk \to \tr\syz^nL\oplus R^{\oplus m} \to \tr\syz^nL' \to 0$ by Lemma \ref{2}.
Induction on the length of $L$ shows that $\tr\syz^nL$ belongs to $\res(\tr\syz^nk)$, which implies
\begin{equation}\label{3}
\res(\tr\syz^nL)\subseteq\res(\tr\syz^nk).
\end{equation}
(Note that this inclusion relation is valid for $L=0$.)
By Proposition \ref{1}(1) the above exact sequence induces an exact sequence
\begin{equation}\label{6}
0 \to \syz\tr\syz^nL' \to \tr\syz^nk\oplus R^{\oplus l} \to \tr\syz^nL\oplus R^{\oplus m} \to 0.
\end{equation}
Let us prove
\begin{equation}\label{5}
\tr\syz^nk\in\res(\tr\syz^nL)
\end{equation}
by induction on $n$.
When $n=0$, the module $\syz\tr\syz^nL'$ is free by Lemma \ref{4}, and \eqref{5} follows from \eqref{6}.
When $n\ge1$, we have:
$$
\syz\tr\syz^nL'\overset{\rm(1)}{\cong}\tr\syz^{n-1}L'\overset{\rm(2)}{\in}\res(\tr\syz^{n-1}k)\overset{\rm(3)}{\subseteq}\res(\tr\syz^{n-1}L)\overset{\rm(4)}{=}\res(\syz\tr\syz^nL)\subseteq\res(\tr\syz^nL).
$$
Here (1),(4) follow from Lemma \ref{4}, (2) from \eqref{3}, and (3) from the induction hypothesis.
Now \eqref{5} is obtained by \eqref{6}.
Consequently, we have the inclusion $\res(\tr\syz^nL)\supseteq\res(\tr\syz^nk)$.
\end{proof}

\begin{lem}\label{9}
Let $M$ be an $R$-module such that $\syz M\in\ext(R\oplus M)$.
Then one has $\ext(R\oplus M)=\res M$.
\end{lem}

\begin{proof}
It is clear that $\ext(R\oplus M)$ is contained in $\res M$.
Let us consider the subcategory
$$
\X=\{\,N\in\ext(R\oplus M)\mid\syz N\in\ext(R\oplus M)\,\}\subseteq\ext(R\oplus M).
$$
By assumption $\X$ contains $R\oplus M$, and we easily see that $\X$ is closed under direct summands.
Let $0\to S\to T\to U\to 0$ be an exact sequence of $R$-modules with $S,U\in\X$.
Then $T,\syz S,\syz U$ are in $\ext(R\oplus M)$, and there is an exact sequence $0\to \syz S\to \syz T\oplus R^{\oplus n}\to \syz U\to 0$.
Hence $\syz T$ belongs to $\ext(R\oplus M)$, which implies $T\in\X$.
Therefore $\X$ is closed under extensions, and thus $\X=\ext(R\oplus M)$.
Now $\ext(R\oplus M)$ is closed under syzygies, which means that it is resolving.
\end{proof}

Now we can achieve the main purpose of this section.

\begin{proof}[{\bf Proof of Theorem \ref{1.1}}]
First of all, it is easy to observe that there is a filtration $\add R=\PD_0^0(R)\subsetneq\PD_0^1(R)\subsetneq\cdots\subsetneq\PD_0^t(R)=\FPD_0(R)$, and that each $\PD_0^n(R)$ is a resolving subcategory of $\mod R$.
When $t=0$, we have $\FPD_0(R)=\PD_0^0(R)=\add R$, which is the smallest resolving subcategory of $\mod R$.
So, we may assume $t>0$.

(1) Let us prove the equalities
$$
\PD_0^n(R)=\res(\tr\syz^{n-1}k)=\ext(R\oplus\textstyle\bigoplus_{i=0}^{n-1}\tr\syz^ik)
$$
for each $1\le n\le t$.
Lemma \ref{4} implies that $\tr\syz^{n-1}k$ belongs to $\PD_0^n(R)$.
Hence $\PD_0^n(R)$ contains $\res(\tr\syz^{n-1}k)$.
Let $M$ be an $R$-module in $\PD_0^n(R)$.
By Proposition \ref{7} we have $\res(\tr\syz^i\tr\syz^iM)\subseteq\res(\tr\syz^{i+1}\tr\syz^{i+1}M\oplus\tr\syz^i\Ext^{i+1}(M,R))$ for $0\le i<n$.
As $\Ext^{i+1}(M,R)$ has finite length, $\tr\syz^i\Ext^{i+1}(M,R)$ is in $\res(\tr\syz^ik)$ by Proposition \ref{8}.
Since $M$ has projective dimension at most $n$, we have $\tr\syz^n\tr\syz^nM=0$.
Lemma \ref{4} implies that $\tr\syz^ik\cong\syz^{n-1-i}(\tr\syz^{n-1}k)\in\res(\tr\syz^{n-1}k)$ for every $0\le i\le n-1$.
Also, $\tr\syz^0\tr\syz^0M$ is isomorphic to $M$ up to free summand.
Hence:
\begin{align*}
\res M &= \res(\tr\syz^0\tr\syz^0M)\subseteq\res(\tr\syz^1\tr\syz^1M\oplus\tr\syz^{n-1}k)\subseteq\res(\tr\syz^2\tr\syz^2M\oplus\tr\syz^{n-1}k)\\
& \subseteq\cdots\subseteq\res(\tr\syz^n\tr\syz^nM\oplus\tr\syz^{n-1}k)=\res(\tr\syz^{n-1}k).
\end{align*}
Therefore $\PD_0^n(R)=\res(\tr\syz^{n-1}k)$ holds.
The equality $\res(\tr\syz^{n-1}k)=\ext(R\oplus\bigoplus_{i=0}^{n-1}\tr\syz^ik)$ is a consequence of Lemmas \ref{4} and \ref{9}.

(2) Let us prove that every resolving subcategory contained in $\FPD_0(R)$ coincides with one of $\PD_0^0(R),\dots,\PD_0^t(R)$.
We start by establishing a claim.
\begin{claim*}
Let $1\le n\le t$.
Let $X$ be an $R$-module in $\PD_0^n(R)\setminus\PD_0^{n-1}(R)$.
Then $\tr\syz^iL$ belongs to $\res X$ for all $R$-modules $L$ of finite length and all integers $0\le i\le n-1$.
\end{claim*}
\begin{proof}[Proof of Claim]
Fix an $R$-module $L$ of finite length.
Set $E=\Ext^n(X,R)$.
Since $X$ has projective dimension $n$, we see that $E$ is a nonzero module of finite length.
Proposition \ref{8} yields $\tr\syz^{n-1}L\in\res(\tr\syz^{n-1}k)=\res(\tr\syz^{n-1}E)$.
It is easy to see that $E\cong\tr\syz^{n-1}X$, whence $\tr\syz^{n-1}E\cong\tr\syz^{n-1}\tr\syz^{n-1}X$.
Thus:
\begin{equation}\label{2201458}
\tr\syz^{n-1}L\in\res(\tr\syz^{n-1}\tr\syz^{n-1}X).
\end{equation}

We show the claim by induction on $n$.
When $n=1$, it follows from \eqref{2201458} and Proposition \ref{1}(4).
Let $n\ge2$.
Since $\syz X\in\PD_0^{n-1}(R)\setminus\PD_0^{n-2}(R)$, the induction hypothesis implies that $\tr\syz^jL$ is in $\res\syz X$ for all $0\le j\le n-2$.
Hence:
\begin{equation}\label{2201558}
\tr\syz^jL\in\res X\quad (0\le j\le n-2).
\end{equation}
By Proposition \ref{7}, there are exact sequences
$$
0 \to \tr\syz^{j+1}\tr\syz^{j+1}X \to \tr\syz^j\tr\syz^jX\oplus R^{\oplus m_j} \to \tr\syz^jE_j \to 0\quad(0\le j\le n-2),
$$
where $E_j=\Ext^{j+1}(X,R)$.
As $E_j$ has finite length, $\tr\syz^jE_j$ is in $\res X$ for $0\le j\le n-2$ by \eqref{2201558}.
Using the above exact sequences, we inductively observe that:
\begin{equation}\label{2201611}
\tr\syz^{n-1}\tr\syz^{n-1}X\in\res X.
\end{equation}
Combining \eqref{2201458}, \eqref{2201558} and \eqref{2201611} yields that $\tr\syz^iL$ belongs to $\res X$ for any integer $0\le i\le n-1$.
\renewcommand{\qedsymbol}{$\square$}
\end{proof}

Now, let $\X\ne\add R$ be a resolving subcategory of $\mod R$ contained in $\FPD_0(R)$.
Then there exists a unique integer $1\le n\le t$ such that $\X$ is contained in $\PD_0^n(R)$ but not contained in $\PD_0^{n-1}(R)$.
We find an $R$-module $X\in\X$ outside $\PD_0^{n-1}(R)$.
The above claim says $\tr\syz^{n-1}k\in\res X$, whence $\tr\syz^{n-1}k\in\X$.
By part (1) of the proof we have $\X=\PD_0^n(R)$.
\end{proof}

We close this section by stating a corollary of Theorem \ref{1.1}, which will be used later.

\begin{cor}\label{219928}
Let $R$ be local.
Let $M\in\FPD_0(R)$.
Then $\res M=\PD_0^n(R)$ with $n=\pd M$.
\end{cor}

\begin{proof}
Theorem \ref{1.1} shows that $\res M=\PD_0^m(R)$ for some integer $m$.
Since $M$ is in $\res M$, we have $n\le m$.
Suppose $n<m$.
Then $M$ is in $\PD_0^{m-1}(R)$.
As $\PD_0^{m-1}(R)$ is resolving, $\res M$ is contained in $\PD_0^{m-1}(R)$.
This is a contradiction, which implies $n=m$.
\end{proof}

\section{Resolving subcategories of modules of finite projective dimension}\label{rsomofpd}

Throughout this section, let $R$ be an arbitrary commutative noetherian ring with identity.
In this section, we study resolving subcategories whose objects are modules of finite projective dimension.

For an $R$-module $M$, denote by $\NF(M)$ the {\em nonfree locus} of $M$, that is, the set of prime ideals $\p$ of $R$ such that $M_\p$ is nonfree as an $R_\p$-module.
As is well-known, $\NF(M)$ is a closed subset of $\Spec R$ (cf. \cite[Corollary 2.11]{res}).

The following result, which is proved in \cite[Lemma 4.6]{radius}, enables us to replace in a resolving subcategory a module with a module whose nonfree locus is irreducible.
This is a generalization of \cite[Theorem 4.3]{res}, and will also be used in the next section.

\begin{prop}\label{c}
Let $M$ be an $R$-module.
For every $\p\in\NF(M)$ there exists $X\in\res M$ satisfying $\NF(X)=V(\p)$ and $\depth X_\q=\inf\{\depth M_\q,\depth R_\q\}$ for all $\q\in V(\p)$.
\end{prop}

Next we state two results which are essentially proved in \cite{stcm}.

\begin{lem}\label{2182108}
Let $\X$ be a resolving subcategory of $\mod R$.
\begin{enumerate}[\rm(1)]
\item
For each $\p\in\Spec R$, the subcategory $\add\X_\p$ of $\mod R_\p$ is resolving.
\item
Let $Z$ be a nonempty finite subset of $\Spec R$.
Let $M$ be an $R$-module such that $M_\p\in\add\X_\p$ for all $\p\in Z$.
Then there exist exact sequences
\begin{align*}
& 0 \to K \to X \to M \to 0,\\
& 0 \to L \to M\oplus K\oplus R^{\oplus n} \to X \to 0
\end{align*}
of $R$-modules with $X\in\X$, $\NF(L)\subseteq\NF(M)$ and $\NF(L)\cap Z=\emptyset$.
\end{enumerate}
\end{lem}

\begin{proof}
Note that the results \cite[Lemmas 4.6 and 4.8]{stcm} hold even if the ring $R$ is not local, as the proofs show.
The first assertion now follows, and the second one can be shown along the same lines as in the proof of \cite[Proposition 4.7]{stcm}.
\end{proof}

Next we investigate the relationship between a module in a resolving subcategory and its localization.
Thanks to the following proposition, we can reduce problems on resolving subcategories to the case where the base ring is local.
(It would be interesting to compare this result with the local-global principle given in \cite{BIK3,BIK2,S}.)

\begin{prop}\label{2182101}
Let $\X$ be a resolving subcategory of $\mod R$.
The following are equivalent for an $R$-module $M$.
\begin{enumerate}[\rm(1)]
\item
$M\in\X$.
\item
$M_\p\in\add\X_\p$ for all $\p\in\Spec R$.
\item
$M_\m\in\add\X_\m$ for all $\m\in\Max R$.
\end{enumerate}
\end{prop}

\begin{proof}
Localization shows (1) $\Rightarrow$ (3) $\Rightarrow$ (2).
As to (2) $\Rightarrow$ (1), assume $M\notin\X$.
Then
$$
\{\,\NF(N)\mid\text{$N\in\mod R\setminus\X$ and $N_\p\in\add\X_\p$ for all $\p\in\Spec R$}\,\}
$$
is a nonempty family of closed subsets of $\Spec R$.
Since $\Spec R$ is noetherian, this set has a minimal element $\NF(L)$; see \cite[Chapter I, Exercise 1.7]{H}.
As $L\notin\X$, we see that $L$ is not projective.
Since projectivity is equivalent to local freeness, $\min\NF(L)$ is nonempty.
Applying Lemma \ref{2182108}(2) to $\min\NF(L)$, we get exact sequences $0 \to K \to X \to L \to 0$ and $0 \to H \to L\oplus K\oplus R^{\oplus n} \to X \to 0$ with $X\in\X$, $\NF(H)\subseteq\NF(L)$ and $\NF(H)\cap\min\NF(L)=\emptyset$.
Since $\min\NF(L)\ne\emptyset$, we have $\NF(H)\subsetneq\NF(L)$.
Fix a prime ideal $\p$ of $R$.
Lemma \ref{2182108}(1) shows that $\add\X_\p$ is a resolving subcategory of $\mod R_\p$.
Localizing the above two exact sequences at $\p$, we observe that $H_\p$ is in $\add\X_\p$.
The minimality of $\NF(L)$ implies that $H$ belongs to $\X$.
It is seen from the second exact sequence above that $L$ is in $\X$, which is again a contradiction.
Consequently, we have $M\in\X$.
\end{proof}

The proposition below, which is a corollary of Theorem \ref{1.1}, will be a basis of the proof of the main result of this section.

\begin{prop}\label{219945}
Let $R$ be local.
Let $M\in\FPD_0(R)$ and $N\in\FPD(R)$.
If $\pd M\le\pd N$, then $M\in\res N$.
\end{prop}

\begin{proof}
We may assume that $M$ is nonfree.
Hence $N$ is also nonfree, and the maximal ideal $\m$ belongs to $\NF(N)$.
Proposition \ref{c} shows that there exists an $R$-module $L\in\res N$ with $\NF(L)=\{\m\}$ and $\depth L=\inf\{\depth N,\depth R\}$.
Since $N$ is of finite projective dimension and $\PD(R)$ is resolving, $\PD(R)$ contains $\res N$.
Hence $L$ is of finite projective dimension as well.
Using the Auslander-Buchsbaum formula, we have
\begin{align*}
\pd L&=\depth R-\depth L=\depth R-\inf\{\depth N,\depth R\}\\
&=\sup\{\pd N,0\}=\pd N\ge\pd M.
\end{align*}
Thus, replacing $N$ with $L$, we may assume that $N$ is in $\FPD_0(R)$.
It follows from Corollary \ref{219928} that $\res N=\PD_0^n(R)$, where $n=\pd N$.
We have $\pd M\le\pd N=n$, and therefore $M$ belongs to $\PD_0^n(R)=\res N$.
\end{proof}

The following theorem is the main result of this section, which gives a criterion for a module in $\FPD(R)$ to belong to a resolving subcategory in $\FPD(R)$.

\begin{thm}\label{4.6}
Let $\X$ be a resolving subcategory of $\mod R$ contained in $\FPD(R)$.
Then the following are equivalent for an $R$-module $M$.
\begin{enumerate}[\rm(1)]
\item
One has $M\in\X$.
\item
For every $\p\in\NF(M)$ there exists $X\in\X$ such that $\pd M_\p\le\pd X_\p$.
\end{enumerate}
\end{thm}

\begin{proof}
The implication (1) $\Rightarrow$ (2) is trivial.
To prove the opposite implication, assume that the condition (2) holds.
By Proposition \ref{2182101} it suffices to show that $M_\m\in\add\X_\m$ for all $\m\in\Max R$.
It holds that for every $\p R_\m\in\NF(M_\m)$ there exists $X_\m\in\add\X_\m$ such that $\pd(M_\m)_{\p R_\m}\le\pd(X_\m)_{\p R_\m}$.
By Lemma \ref{2182108}(1) the subcategory $\add\X_\m$ of $\mod R_\m$ is resolving.
Hence we may assume that $(R,\m)$ is local.
Let us deduce (1) by induction on $m:=\dim\NF(M)$.

When $m=-\infty$, the module $M$ is free, and belongs to $\X$.
When $m=0$, we have $\NF(M)=\{\m\}$.
Hence $M$ is in $\FPD_0(R)$, and by (2) there exists $X\in\X$ with $\pd M\le\pd X$.
Proposition \ref{219945} implies that $M$ belongs to $\res X$, and thus $M\in\X$.

Now, let $m\ge1$.
Let $\p\in\min\NF(M)$.
Then we have $\dim\NF(M_\p)=0$.
Lemma \ref{2182108}(1) and the basis of the induction show that $M_\p\in\add\X_\p$.
Hence we can apply Lemma \ref{2182108}(2) to the nonempty finite set $\min\NF(M)$ (since $\NF(M)$ is closed and in general there are only finitely many minimal primes over an ideal, the set $\min\NF(M)$ is finite), and obtain exact sequences
\begin{align}
& 0 \to K \to Y \to M \to 0, \\
&\label{2191047} 0 \to L \to M\oplus K\oplus R^{\oplus l} \to Y \to 0
\end{align}
with $Y\in\X$, $\NF(L)\subseteq\NF(M)$ and $\NF(L)\cap\min\NF(M)=\emptyset$.
We easily see that $L\in\FPD(R)$ and that $\dim\NF(L)<m$.

Let $\p\in\NF(L)$.
Then $\p$ is in $\NF(M)$, and hence $\pd M_\p\le\pd W_\p$ for some $W\in\X$ by (2).
There are exact sequences $0 \to K_\p \to Y_\p \to M_\p \to 0$ and $0 \to L_\p \to M_\p\oplus K_\p\oplus R_\p^{\oplus l} \to Y_\p \to 0$, which yield $\pd L_\p\le\sup\{\pd M_\p,\pd Y_\p\}$.
Thus the module $L$ satisfies the condition (2).
Applying the induction hypothesis to $L$, we get $L\in\X$.
By \eqref{2191047} we observe $M\in\X$.
\end{proof}

Finally we recall the statement of Theorem \ref{1.4} from the Introduction and give a proof:

\begin{thmd}
Let $R$ be a commutative noetherian ring.
Suppose that $\Tor$-rigidity holds for $\PD(R_\p)$ for all $\p \in \Spec R$.
The following are equivalent for $M,N \in \FPD(R)$:
\begin{enumerate}[\rm(1)]
\item
$\pd M_{\p} \leq\max\{\pd N_{\p},0\}$ for all $\p \in \Spec R$.
\item
$M\in\res N$.
\item
$\Supp \Tor_i^R(M,X) \subseteq \Supp \Tor_i^R(N,X)$ for all $i>0$ and all $X\in \mod R$.
\end{enumerate}
\end{thmd}

\begin{proof}[{\bf Proof of Theorem \ref{1.4}}]
The implication $(1) \Rightarrow (2)$ is a consequence of Theorem \ref{4.6}.
Indeed, every prime ideal $\p\in\NF(M)$ satisfies $\pd M_\p\le\pd N_\p$ by (1).

As to the implication $(3) \Rightarrow (1)$, assume (3) and let $X=R/\p$.
Then it is seen for each $i>0$ that $\Tor_i^{R_\p}(N_\p,\kappa(\p))=0$ implies $\Tor_i^{R_\p}(M_\p,\kappa(\p))=0$.
Thus (1) follows.

Assume $(2)$ and pick a prime ideal $\p$ which is not in $\Supp \Tor_i^R(N,X)$.
Then $\Tor_i^{R_{\p}} (N_{\p}, X_{\p})= 0$, thus  $\Tor_j^{R_{\p}} (N_{\p}, X_{\p})= 0$ for all $j\geq i$ by $\Tor$-rigidity.
Since $ M_\p \in \res N_\p$, it is easily seen that $\Tor_i^{R_{\p}} (M_{\p}, X_{\p})= 0$.
Therefore $\p$ is not in $\Supp\Tor_i^R(M,X)$.
\end{proof}

\section{Dominant resolving subcategories over a Cohen-Macaulay ring}\label{drsoacr}

In this section, we consider dominant resolving subcategories over a Cohen-Macaulay ring.
In the first three results below, we investigate, over a Cohen-Macaulay local ring, the structure of resolving closures containing syzygies of the residue field.

\begin{lem}\label{1'}
Let $(R,\m,k)$ be a Cohen-Macaulay local ring.
Let $\X$ be a resolving subcategory of $\mod R$.
Suppose that $\X$ contains an $R$-module of depth $0$.
If $\syz^nk$ is in $\X$ for some $n\ge0$, then $k$ is in $\X$.
\end{lem}

\begin{proof}
Let $X$ be an $R$-module in $\X$ with $\depth X=0$.
If $\NF(X)=\emptyset$, then $X$ is free.
If $\NF(X)\ne\emptyset$, then $\m$ belongs to it, and applying Proposition \ref{c} to $\m$, we find a module $X'\in\res X$ with $\depth X'=0$ and $\NF(X')=\{\m\}$.
Thus, we may assume that $X$ is locally free on the punctured spectrum of $R$.
As $\depth X=0$, there is an exact sequence
\begin{equation}\label{11302156}
0 \to k \to X \to C \to 0,
\end{equation}
and by Proposition \ref{1}(1) we get an exact sequence $0 \to \syz C \to k\oplus R^{\oplus m} \to X \to 0$.
Fix an integer $i\ge1$.
Applying Proposition \ref{1}(1) ($(i-1)$ times) gives an exact sequence $0 \to \syz^iC \to \syz^{i-1}k\oplus R^{\oplus l} \to \syz^{i-1}X \to 0$ for some $l\ge0$.
Since $C$ is locally free on the punctured spectrum by \eqref{11302156}, it is in $\res_Rk$ by \cite[Theorem 2.4]{stcm}, whence $\syz^iC$ is in $\res_R(\syz^ik)$.
Thus, $\syz^ik\in\X$ implies $\syz^{i-1}k\in\X$ for each $i\ge1$.
The assertion follows.
\end{proof}

\begin{prop}\label{5'}
Let $(R,\m,k)$ be a $d$-dimensional Cohen-Macaulay local ring.
Let $M$ be an $R$-module of depth $t$.
Then $\syz^tk$ belongs to $\res(M\oplus\syz^nk)$ for all $n\ge0$.
\end{prop}

\begin{proof}
As $\res_R(M\oplus\syz_R^{i+1}k)\subseteq\res_R(M\oplus\syz_R^ik)$ for each $i\ge0$, we may assume $n\ge d$.
Then $\syz_R^nk$ is a maximal Cohen-Macaulay $R$-module.
Let $\xx=x_1,\dots,x_t$ be a sequence of elements of $R$ having the following properties:
\begin{enumerate}[(1)]
\item
The sequence $\xx$ is regular on both $M$ and $R$,
\item
The element $\overline{x_i}\in\m/(x_1,\dots,x_{i-1})$ is not in $(\m/(x_1,\dots,x_{i-1}))^2$ for $1\le i\le t$.
\end{enumerate}
We can actually get such a sequence by choosing an element of $\m$ outside the ideal $\m^2+(x_1,\dots,x_{i-1})$ and the prime ideals in $\Ass_RM/(x_1,\dots,x_{i-1})M\cup\Ass_RR/(x_1,\dots,x_{i-1})$ for each $1\le i\le t$.
Putting $N=M\oplus\syz_R^nk$, we see that $\xx$ is an $N$-sequence.
Applying \cite[Corollary 5.3]{syz2} repeatedly, we have an isomorphism
$$
\syz_R^nk/\xx\syz_R^nk\cong\bigoplus_{j=0}^t(\syz_{R/(\xx)}^{n-j}k)^{\oplus\binom{t}{j}}.
$$
Hence $\res_{R/(\xx)}(N/\xx N)=\res_{R/(\xx)}(M/\xx M\oplus\syz_R^nk/\xx\syz_R^nk)$ contains $\syz_{R/(\xx)}^nk$.
Since $\res_{R/(\xx)}(N/\xx N)$ contains the module $M/\xx M$ of depth zero, it contains $k$ by Lemma \ref{1'}.
It follows that $\syz_R^tk$ belongs to $\res_R(\syz^t_R(N/\xx N))$.
The Koszul complex of $\xx$ with respect to $N$ gives an exact sequence
$$
0 \to N^{\oplus\binom{t}{t}} \to N^{\oplus\binom{t}{t-1}} \to \cdots \to N^{\oplus\binom{t}{1}} \to N^{\oplus\binom{t}{0}} \to N/\xx N \to 0,
$$
which shows $\syz^t_R(N/\xx N)\in\res_RN$ by \cite[Lemma 4.3]{stcm}.
Therefore $\syz_R^tk\in\res_RN$.
\end{proof}

\begin{cor}\label{2'}
Let $R$ be a Cohen-Macaulay local ring with residue field $k$.
Let $M$ be an $R$-module of depth $t$ that is locally free on the punctured spectrum of $R$.
Then $\res(M\oplus\syz^nk)=\res(\syz^tk\oplus\syz^nk)$ for all $n\ge0$.
\end{cor}
\begin{proof}
By \cite[Theorem 2.4]{stcm} the module $M$ belongs to $\res\syz^tk$, which gives the inclusion $\res(M\oplus\syz^nk)\subseteq\res(\syz^tk\oplus\syz^nk)$.
The other inclusion follows from Proposition \ref{5'}.
\end{proof}

We slightly extend the definition of a dominant resolving subcategory as follows.

\begin{dfn}
Let $W$ be a subset of $\Spec R$, and let $\X$ be a resolving subcategory of $\mod R$.
We say that $\X$ is {\em dominant} on $W$ if for all $\p\in W$ there exists $n\ge0$ such that $\syz^n\kappa(\p)\in\add\X_\p$.
(An dominant resolving subcategory of $\mod R$ is nothing but a resolving subcategory that is dominant on $\Spec R$.)
\end{dfn}

The main result of this section is the following theorem, which yields a criterion for a module to be in a dominant resolving subcategory.

\begin{thm}\label{g}
Let $R$ be a Cohen-Macaulay ring.
Let $\X$ be a resolving subcategory of $\mod R$.
Let $M$ be an $R$-module such that $\X$ is dominant on $\NF(M)$.
Then the following are equivalent:
\begin{enumerate}[\rm(1)]
\item
$M$ belongs to $\X$.
\item
$\depth M_\p\ge\min_{X\in\X}\{\depth X_\p\}$ holds for all $\p\in\NF(M)$.
\end{enumerate}
\end{thm}

\begin{proof}
The implication $(1)\Rightarrow(2)$ is trivial.
To show $(2)\Rightarrow(1)$, we may assume that $R$ is local with maximal ideal $\m$ by Lemma \ref{2182108}(1) and Proposition \ref{2182101}.
Let $M$ be an $R$-module with $\depth M_\p\ge\min_{X\in\X}\{\depth X_\p\}$ for all $\p\in\Spec R$.
We show by induction on $m:=\dim\NF(M)$ that $M$ belongs to $\X$.
When $m=-\infty$, the module $M$ is projective, and is in $\X$.
When $m=0$, we have $\NF(M)=\{\m\}$.
By assumption, $\syz^nk$ is in $\X$ for some $n\ge0$.
As $M$ is locally free on the punctured spectrum, it follows from Corollary \ref{2'} that $\res(M\oplus\syz^nk)=\res(\syz^tk\oplus\syz^nk)$, where $t=\depth M$.
We have $t=\depth M\ge\min_{X\in\X}\{\depth X\}$, which gives an $R$-module $X\in\X$ with $t\ge\depth X=:s$.
Hence $t-s\ge0$, and $\syz^tk=\syz^{t-s}(\syz^sk)$ belongs to $\res\syz^sk$.
Proposition \ref{5'} implies that $\syz^sk$ is in $\res(X\oplus\syz^nk)$, and we have
$$
M\in\res(M\oplus\syz^nk)=\res(\syz^tk\oplus\syz^nk)\subseteq\res(\syz^sk\oplus\syz^nk)\subseteq\res(X\oplus\syz^nk)\subseteq\X.
$$

Let us consider the case $m\ge1$.
This is indeed handled very similarly to the proof of Theorem \ref{4.6}.
Using Lemma \ref{2182108}(1) and the basis of the induction, we observe that $M_\p\in\add\X_\p$ for all $\p\in\min\NF(M)$.
Lemma \ref{2182108}(2) implies that there exist exact sequences
\begin{align}
& 0 \to K \to X \to M \to 0, \\
\label{1645}& 0 \to L \to M\oplus K\oplus R^{\oplus l} \to X \to 0
\end{align}
with $X\in\X$, $\NF(L)\subseteq\NF(M)$ and $\dim\NF(L)<m$.
Localizing these exact sequences, we obtain $\depth L_\p\ge\min_{Y\in\X}\{\depth Y_\p\}$ for all $\p\in\NF(L)$.
Now we can apply the induction hypothesis to $L$ to get $L\in\X$.
By \eqref{1645}, the module $M$ belongs to $\X$.
\end{proof}

The following result is a direct consequence of Theorem \ref{g}.

\begin{cor}\label{2192304}
Let $R$ be a Cohen-Macaulay ring with $\dim R=d<\infty$.
Let $\X$ be a resolving subcategory of $\mod R$.
The following are equivalent:
\begin{enumerate}[\rm(1)]
\item
$\X$ is dominant;
\item
For all $\p\in\Spec R$, there exists $n\ge0$ such that $\syz^n(R/\p)\in\X$;
\item
For all $\p\in\Spec R$, $\syz^d(R/\p)\in\X$.
\end{enumerate}
\end{cor}

\begin{proof}
The implications (3) $\Rightarrow$ (2) $\Rightarrow$ (1) are obvious.
To show (1) $\Rightarrow$ (3), let $\q\in\Spec R$.
Evidently $\syz^d(R/\p)_\q$ is a maximal Cohen-Macaulay $R_\q$-module.
Hence we have $\depth\syz^d(R/\p)_\q\ge\dim R_\q\ge\min_{X\in\X}\{\depth X_\q\}$, and $\syz^d(R/\p)\in\X$ by Theorem \ref{g}.
\end{proof}

\section{Proofs of Theorems \ref{1.2}\ and \ref{1.3}}\label{mainproofs}

This section is devoted to proving Theorem \ref{1.2} and \ref{1.3} by using the results obtained in the previous two sections.
First of all, we define subcategories determined by local finiteness of homological dimensions.

\begin{nota}
We denote by $\h$ either projective dimension $\pd$ or Cohen-Macaulay dimension $\CMdim$.
({\em Cohen-Macaulay dimension} is defined in \cite[Definitions 3.2 and 3.2']{G}.
We do not remind the reader of its definition because it requires a lot of words which are not important here.
The only importance for us concerning this dimension is concentrated in Lemma \ref{5.2} below.)
Let $\FD_{\h}(R)$ denote the subcategory of $\mod R$ consisting of modules $M$ with $\h_{R_\p}(M_\p)<\infty$ for all $\p\in\Spec R$.
\end{nota}

Here are some basic properties of projective and Cohen-Macaulay dimension and their consequences, which will often be used without reference.

\begin{lem}\label{5.2}
The following hold for $(\h,\P)=(\pd,{\rm regular}),(\CMdim,{\rm Cohen}$-${\rm Macaulay})$.
\begin{enumerate}[\rm(1)]
\item
Suppose that $R$ is local.
\begin{enumerate}[\rm(a)]
\item
It holds for an $R$-module $M$ that $\h(M)\in\{-\infty\}\cup\N\cup\{\infty\}$, and that $\h(M)=-\infty\Leftrightarrow M=0$.
\item
If $M,N$ are $R$-modules with $M\cong N$, then $\h(M)=\h(N)$.
\item
For an $R$-module $M$ one has $\CMdim M\le\pd M$.
Equality holds if $\pd M<\infty$.
\item
If $M$ is an $R$-module and $N$ is a direct summand of $M$, then $\h(N)\le\h(M)$.
\item
If $M$ is an $R$-module with $\h(M)<\infty$, then $\h(M)=\depth R-\depth M$.
\item
For a nonzero $R$-module $M$, one has $\h(\syz^nM)=\sup\{\h(M)-n,0\}$.
\item
If $R$ satisfies $\P$, then $\h(M)<\infty$ for every $R$-module $M$.
\end{enumerate}
\item
For $M\in\mod R$ and $\p,\q\in\Spec R$ with $\p\subseteq\q$, one has $\h_{R_\p}(M_\p)\le\h_{R_\q}(M_\q)$.
\item
If $R$ satisfies $\P$, then $\FD_{\h}(R)=\mod R$.
\item
One has $\FD_{\pd}(R)=\FPD(R)$.
\item
For every $R$-module $M\in\FD_{\h}(R)$, the inequality $\h(M_\p)\le\grade\p$ holds.
\end{enumerate}
\end{lem}

\begin{proof}
The statements (1)--(3) are well-known for $\pd$, and those for $\CMdim$ are stated in \cite[\S3]{G}.
The statement (4) follows from \cite[4.5]{BM}.
Let us show the statement (5).
There is an equality $\grade\p=\inf\{\,\depth R_\q\mid\q\in V(\p)\,\}$ by \cite[Proposition 1.2.10]{BH}, so we have $\grade\p=\depth R_\q$ for some $\q\in V(\p)$.
Since $\h(M_\q)$ is finite, it holds that $\h(M_\p)\le\h(M_\q)=\depth R_\q-\depth M_\q\le\depth R_\q=\grade\p$.
\end{proof}

The following lemma will be necessary in both of the proofs of Theorems \ref{1.2} and \ref{1.3}.

\begin{lem}\label{h}
Let $f$ be a grade consistent function on $\Spec R$.
(See Definition \ref{gcdef} for the definition of a grade consistent function.)
For each $\p\in\Spec R$ there is an $R$-module $E$ satisfying $\pd E<\infty$, $\pd E_\q\le f(\q)$ for all $\q\in\Spec R$ and $\pd E_\p=f(\p)$.
\end{lem}

\begin{proof}
Put $n=\grade\p$.
If $n=0$, then $f(\p)=0$, and we can take $E:=R$.
Let $n\ge1$.
Choose an $R$-sequence $\xx=x_1,\dots,x_n$ in $\p$.
It is easy to observe that $\p$ belongs to $\NF(R/(\xx))$.
Proposition \ref{c} implies that there exists an $R$-module $X\in\res(R/(\xx))$ with $\NF(X)=V(\p)$ and $\depth X_\q=\inf\{\depth R_\q/\xx R_\q,\depth R_\q\}=\depth R_\q-n$ for all $\q\in V(\p)$.
Note that $n-f(\p)\ge0$.
Put $E:=\syz^{n-f(\p)}X$.
As the $R$-module $R/(\xx)$ has finite projective dimension, so does $X$, and so does $E$.
If $\q$ is not in $V(\p)$, then $\q\notin\NF(X)$.
Hence $X_\q$ is $R_\q$-free, and so is $E_\q$.
Thus we may assume $\q\in V(\p)$.
Then we have $\pd E_\q=\sup\{\pd X_\q-n+f(\p),0\}=\sup\{\depth R_\q-\depth X_\q-n+f(\p),0\}=\sup\{f(\p),0\}=f(\p)\le f(\q)$.
In particular, it holds that $\pd E_\p=f(\p)$.
\end{proof}

Now we are on the stage to achieve the main purpose of this section.
For the convenience of the reader, we restate the statements of Theorems \ref{1.2} and \ref{1.3}.

\begin{thma}
Let $R$ be a commutative noetherian ring.
There is a 1-1 correspondence
$$
\begin{CD}
\left\{
\begin{matrix}
\text{Resolving subcategories of $\mod R$}\\
\text{contained in $\FPD(R)$}
\end{matrix}
\right\}
\begin{matrix}
@>{\phi}>>\\
@<<{\psi}<
\end{matrix}
\left\{
\begin{matrix}
\text{Grade consistent functions}\\
\text{on $\Spec R$}
\end{matrix}
\right\}.
\end{CD}
$$
Here $\phi,\psi$ are defined by $\phi(\X)(\p)=\max_{X\in\X}\{\pd X_\p\}$ and $\psi(f)=\{\,M\in\mod R\mid\text{$\pd M_\p\le f(\p)$ for all $\p\in\Spec R$}\,\}$.
\end{thma}

\begin{thmb}
Let $R$ be a Cohen-Macaulay ring.
Then one has a 1-1 correspondence
$$
\begin{CD}
\left\{
\begin{matrix}
\text{Dominant resolving subcategories}\\
\text{of $\mod R$}\\
\end{matrix}
\right\}
\begin{matrix}
@>{\phi}>>\\
@<<{\psi}<
\end{matrix}
\left\{
\begin{matrix}
\text{Grade consistent functions}\\
\text{on $\Spec R$}
\end{matrix}
\right\},
\end{CD}
$$
where $\phi,\psi$ are defined by $\phi(\X)(\p)=\height\p-\min_{X\in\X}\{\depth X_\p\}$ and $\psi(f)=\{\,M\in\mod R\mid\text{$\depth M_\p\ge\height\p-f(\p)$ for all $\p\in\Spec R$}\,\}$.
\end{thmb}

\begin{proof}[{\bf Proof of Theorem \ref{1.2}}]
It is easy to verify that $\phi,\psi$ are well-defined maps (use Lemma \ref{5.2}).
Let $\X$ be a resolving subcategory contained in $\FPD(R)$.
Then
$$
\psi\phi(\X)=\{\,M\in\mod R\mid\text{$\pd M_\p\le\max_{X\in\X}\{\pd X_\p\}$ for all $\p\in\Spec R$}\,\}\supseteq\X.
$$
It follows from Theorem \ref{4.6} that $\psi\phi(\X)=\X$.
Let $f$ be a grade consistent function on $\Spec R$.
Fix a prime ideal $\p$.
We have
$$
\phi\psi(f)(\p)=\max\{\,\pd X_\p\mid\text{$X\in\mod R$ with $\pd X_\q\le f(\q)$ for all $\q\in\Spec R$}\,\}\le f(\p).
$$
It follows from Lemma \ref{h} that $\phi\psi(f)(\p)=f(\p)$.
Therefore $\phi\psi(f)=f$.
\end{proof}

\begin{proof}[{\bf Proof of Theorem \ref{1.3}}]
We see from \cite[Theorem 1.1]{AIL} that for every $R$-module $M$ the {\em large restricted flat dimension} 
$$
\Rfd_RM:=\sup_{\p\in\Spec R}\{\depth R_\p-\depth M_\p\}
$$
of $M$ is finite.
Fix $\p\in\Spec R$ and set $n:=\Rfd_R(R/\p)<\infty$.
It follows from \cite[Proposition (2.3) and Theorem (2.8)]{CFF} that $\CMdim_{R_\q}(R_\q/\p R_\q)-n=\Rfd_{R_\q}(R_\q/\p R_\q)-n\le0$, and hence $\CMdim_{R_\q}\syz^n(R_\q/\p R_\q)=\sup\{\CMdim_{R_\q}(R_\q/\p R_\q)-n,0\}=0\le f(\q)$ for each $\q\in\Spec R$.
Thus $\syz^n(R/\p)\in\psi(f)$, which implies that $\psi(f)$ is dominant.
Now we easily check that $\phi,\psi$ are well-defined.
Let $\X$ be a dominant resolving subcategory of $\mod R$.
Then
$$
\psi\phi(\X)=\{\,M\in\mod R\mid\text{$\depth M_\p\ge\min_{X\in\X}\{\depth X_\p\}$ for all $\p\in\Spec R$}\,\}\supseteq\X.
$$
By virtue of Theorem \ref{g}, the equality $\psi\phi(\X)=\X$ holds.
Let $f$ be a grade consistent function on $\Spec R$, and let $\p\in\Spec R$.
Then
\begin{align*}
\phi\psi(f)(\p)
&=\max\{\,\CMdim X_\p\mid\text{$X\in\mod R$ with $\CMdim X_\q\le f(\q)$ for all $\q\in\Spec R$}\,\}\\
&\le f(\p).
\end{align*}
Taking $E$ as in Lemma \ref{h}, we deduce $\phi\psi(f)(\p)=f(\p)$.
Thus $\phi\psi(f)=f$.
\end{proof}

Notice that all resolving subcategories over a regular ring $R$ is dominant since a high syzygy of an $R$-module is projective.
Thus as a common consequence of Theorems \ref{1.2} and \ref{1.3} we get a complete classification of the resolving subcategories over a regular ring:

\begin{cor}
Let $R$ be a regular ring.
Then there exist mutually inverse bijections
$$
\begin{CD}
\left\{
\begin{matrix}
\text{Resolving subcategories}\\
\text{of $\mod R$}
\end{matrix}
\right\}
\begin{matrix}
@>{\phi}>>\\
@<<{\psi}<
\end{matrix}
\left\{
\begin{matrix}
\text{Grade consistent functions}\\
\text{on $\Spec R$}
\end{matrix}
\right\}.
\end{CD}
$$
Here $\phi,\psi$ are defined by $\phi(\X)=[\,\p\mapsto\max_{X\in\X}\{\pd X_\p\}\,]$ and $\psi(f)=\{\,M\in\mod R\mid\text{$\pd M_\p\le f(\p)$ for all $\p\in\Spec R$}\,\}$.
\end{cor}

Recall that a local ring $R$ has (at most) an {\em isolated singularity} if $R_\p$ is a regular local ring for each nonmaximal prime ideal $\p$ of $R$.
As an application of Theorem \ref{1.3}, we have the following.

\begin{cor}
Let $(R,\m,k)$ be a $d$-dimensional Cohen-Macaulay local ring with an isolated singularity.
Then one has the following one-to-one correspondences.
$$
\begin{CD}
\left\{
\begin{matrix}
\text{Resolving subcategories of $\mod R$}\\
\text{containing $\syz^dk$}
\end{matrix}
\right\}
\begin{matrix}
@>{\phi}>>\\
@<<{\psi}<
\end{matrix}
\left\{
\begin{matrix}
\text{Grade consistent functions}\\
\text{on }\Spec R
\end{matrix}
\right\},
\end{CD}
$$
where $\phi,\psi$ are defined by $\phi(\X)=[\,\p\mapsto\height\p-\min_{X\in\X}\{\depth X_\p\}\,]$ and $\psi(f)=\{\,M\in\mod R\mid\text{$\depth M_\p\ge\height\p-f(\p)$ for all $\p\in\Spec R$}\,\}$.
\end{cor}

\begin{proof}
According to Theorem \ref{1.3}, it is enough to check that a resolving subcategory $\X$ of $\mod R$ is dominant if and only if $\syz^dk\in\X$.
The `only if' part is obvious.
As to the `if' part, take $\p\in\Spec R$.
Let us show $\syz^d\kappa(\p)\in\add\X_\p$.
It is clear if $\p=\m$, so assume $\p\ne\m$.
Then $R_\p$ is regular, and $\syz^d\kappa(\p)$ is a free $R_\p$-module.
Hence it is in $\add\X_\p$.
\end{proof}

\section{Some examples}\label{exa}

In this section, we state some examples to cultivate a better understanding and show the importance of the results we have obtained in the preceding sections.
First of all, let us give examples of a grade consistent function.

\begin{ex}\label{12011405}
Define an $\N$-valued function $f$ on $\Spec R$ by for each $\p\in\Spec R$:
\begin{enumerate}[(1)]
\item
$f(\p)=0$,
\item
$f(\p)=\grade\p$, or
\item
$f(\p)=\begin{cases}
0 & \text{if }\p=0,\\
1 & \text{if }\p\ne0,
\end{cases}$ assuming that $R$ is a domain.
\end{enumerate}
Then $f$ is grade consistent.
The map $\psi$ in Theorem \ref{1.2} sends these grade consistent functions to the resolving subcategories
$$
\add R,\quad\PD(R),\quad\{ M\in\mod R\mid \pd M\le1\}
$$
contained in $\PD(R)$, and the map $\psi$ in Theorem \ref{1.3} sends them to the dominant resolving subcategories
$$
\CM(R),\quad\mod R,\quad\{M\in\mod R\mid\depth M\ge\dim R-1\},
$$
respectively.
\end{ex}

Next, given two modules $M$ and $N$, let us consider when $M$ belongs to the resolving closure of $N$.

\begin{ex}\label{12011208}
Let $R=k[x,y]$ be a polynomial ring over a field $k$.
Then one has:
\begin{enumerate}[(1)]
\item
$(x^2,y)\in\res(R/(x,y))$.
\item
$R/(xy)\notin\res(R/(x))$.
\item
$R/(x)\in\res(R/(xy))$.
\end{enumerate}
\end{ex}

We here give direct proofs of these three statements by actually constructing direct summands, extensions and syzygies.

\begin{proof}[Proof of Example \ref{12011208}]
(1) There is a natural exact sequence
$$
0 \to (x,y)/(x^2,y)\to R/(x^2,y)\to R/(x,y)\to0
$$
of $R$-modules.
It is easily seen that $(x,y)/(x^2,y)$ is isomorphic to $R/(x,y)$.
Hence $R/(x^2,y)$ is in $\res(R/(x,y))$, and so is its first syzygy $(x^2,y)$.

(2) Suppose that $R/(xy)$ belongs to $\res_R(R/(x))$.
Then localization at the prime ideal $(y)$ of $R$ shows that $(R/(xy))_{(y)}=R_{(y)}/yR_{(y)}$ belongs to $\res_{R_{(y)}}((R/(x))_{(y)})=\add_{R_{(y)}}(R_{(y)})$ (cf. \cite[Proposition 1.11]{stcm}).
Hence $R_{(y)}/yR_{(y)}$ is free as an $R_{(y)}$-module, which is a contradiction.

(3) We have an exact sequence
\begin{equation}\label{12011206}
0 \to R/(xy) \xrightarrow{f} R/(x)\oplus R/(xy^2) \xrightarrow{g} R/(xy) \to 0
\end{equation}
of $R$-modules, where $f$ and $g$ are defined by $f(\overline a)=\binom{\overline a}{\overline{ay}}$ and $g(\binom{\overline b}{\overline c})=\overline{by-c}$ for $a,b,c\in R$.
We observe from this exact sequence that $R/(x)$ belongs to $\res(R/(xy))$.
\end{proof}

Thus we have got proofs of the three statements in Example \ref{12011208}, but such statements are difficult to show in general.
Especially, it is pretty hard in general to find such an exact sequence as \eqref{12011206}.

Our Theorems \ref{4.6} and \ref{g} are useful to verify such containment.
To see this, we first transform them into a more reasonable form by using \cite[Propositions 1.11 and 1.12]{stcm}:

\begin{prop}\label{12011550}
Let $M,N$ be $R$-modules.
\begin{enumerate}[\rm(1)]
\item
Suppose that $M$ and $N$ are in $\FPD(R)$.
Then $M\in\res N$ if and only if $\pd M_\p\le\pd N_\p$ for all $\p\in\NF(M)$.
\item
Suppose that $R$ is a $d$-dimensional Cohen-Macaulay local ring with residue field $k$ having an isolated singularity.
Then $M\in\res(N\oplus\syz^dk)$ if and only if $\depth M_\p\ge\depth N_\p$ for all $\p\in\NF(M)$.
\end{enumerate}
\end{prop}

Applying Proposition \ref{12011550}(1), we can show the three statements in Example \ref{12011208} easily as follows.
We should remark that it is unnecessary to construct direct summands, extensions and syzygies; only necessary is to calculate nonfree loci and projective dimensions.

\begin{proof}[Another proof of Example \ref{12011208}]
Use Proposition \ref{12011550}(1) and the following observations:

(1) One has $\NF(x^2,y)=\{(x,y)\}$ and $\pd(x^2,y)_{(x,y)}=1\le2=\pd(R/(x,y))_{(x,y)}$.

(2) One has $(y)\in\NF(R/(xy))$ and $\pd(R/(xy))_{(y)}=1>-\infty=\pd(R/(x))_{(y)}$.

(3) If $\p\in\NF(R/(x))$, then $x\in\p$, and hence $\pd(R/(x))_\p=1=\pd(R/(xy))_\p$.
\end{proof}

The following is an example of an application of Proposition \ref{12011550}(2).

\begin{ex}\label{12011402}
Let $R=k[[x,y]]/(xy)$.
Then $R/(x)$ belongs to $\res(R/(x-y)\oplus(x,y))$.
\end{ex}

\begin{proof}
As $\NF(R/(x))=\{(x,y)\}$, it is enough to check that $\depth R/(x)\ge\depth R/(x-y)$ by Proposition \ref{12011550}(2).
This is trivial because $\depth R/(x-y)=0$.
\end{proof}

In fact, the assertion of Example \ref{12011402} can be generalized to the following form by applying our classification Theorems \ref{1.2} and \ref{1.3}.

\begin{prop}
Let $R$ be a Cohen-Macaulay local ring of dimension one.
Then:
\begin{enumerate}[\rm(1)]
\item
The resolving subcategories of $\mod R$ contained in $\PD(R)$ are $\add R$ and $\PD(R)$.
\item
The dominant resolving subcategories of $\mod R$ are $\CM(R)$ and $\mod R$.
\end{enumerate}
\end{prop}

\begin{proof}
In view of Theorems \ref{1.2} and \ref{1.3} combined with Example \ref{12011405}, we have only to show that the grade consistent functions on $\Spec R$ are $\zeta$ and $\gamma$, which are defined by $\zeta(\p)=0$ and $\gamma(\p)=\grade\p$ for every $\p\in\Spec R$.

Let $f$ be any grade consistent function on $\Spec R$.
Since $R$ is a $1$-dimensional local ring, the spectrum $\Spec R$ consists of the minimal prime ideals $\p_1,\dots,\p_n$ and the maximal ideal $\m$.
The fact that $0\le f(\p)\le\grade\p$ for all $\p\in\Spec R$ shows that $f(\p_i)=0$ for each $1\le i\le n$ and that $f(\m)$ is either $0$ or $1$ (as $R$ is Cohen-Macaulay, $f(\m)$ can be $1$).
The function $f$ coincides with $\zeta$ if $f(\m)=0$ and with $\gamma$ if $f(\m)=1$.
\end{proof}

\section{Classification for locally hypersurface rings}\label{reghyper}

In this section, we show that over a locally complete intersection ring, a resolving subcategory is determined by its ``maximal Cohen-Macaulay" and ``finite projective dimension" parts. As a consequence, a proof of Theorem \ref{hyper} can be given.
We begin with showing two lemmas.

\begin{lem}\label{371445}
Let $R$ be a commutative noetherian ring.
\begin{enumerate}[\rm(1)]
\item
Let $0 \to L\oplus L' \xrightarrow{(f,g)} M \to N \to 0$ be an exact sequence of $R$-modules.
If $\Ext_R^1(N,L)=0$, then $f$ is a split monomorphism.
\item
Let $0 \to L \to M \xrightarrow{\binom{f}{g}} N\oplus N' \to 0$ be an exact sequence of $R$-modules.
If $\Ext_R^1(N,L)=0$, then $f$ is a split epimorphism.
\end{enumerate}
\end{lem}

\begin{proof}
We only prove the first assertion, as the dual argument shows the second assertion.
We have an exact sequence $0\to L'\xrightarrow{g}M\xrightarrow{h}K\to0$, and it is seen that there is an exact sequence $0\to L\xrightarrow{hf}K\to N\to0$.
This splits as $\Ext_R^1(N,L)=0$, which means that there is a homomorphism $r:K\to L$ such that $rhf=1$.
Hence $f$ is a split monomorphism.
\end{proof}

\begin{lem}\label{371446}
Let $R$ be a commutative noetherian ring.
Let $\Y,\ZZ$ be resolving subcategories of $\mod R$ with $\Ext_R^1(Z,Y)=0$ for all $Y\in\Y$ and $Z\in\ZZ$.
Then $\res(\Y\cup\ZZ)$ coincides with the subcategory of $\mod R$ consisting of modules $M$ admitting an exact sequence $0 \to Z \to M\oplus N \to Y \to 0$ with $Y\in\Y$ and $Z\in\ZZ$.
\end{lem}

\begin{proof}
Let $\X$ denote the subcategory of $\mod R$ consisting of modules $M$ admitting an exact sequence $0 \to Z \to M\oplus N \to Y \to 0$ with $Y\in\Y$ and $Z\in\ZZ$.
Clearly, $\X$ is contained in $\res(\Y\cup\ZZ)$.
Let us show the opposite inclusion.
As a resolving subcategory contains the zero module, we see that $\X$ contains $\Y\cup\ZZ$.
In particular, $\X$ contains $R$.
If we have an exact sequence $0 \to Z \to M\oplus N \to Y \to 0$ of $R$-modules with $Y\in\Y$ and $Z\in\ZZ$, then there is an exact sequence $0 \to \syz Z \to \syz M\oplus \syz N \to \syz Y \to 0$, and $\syz Y,\syz Z$ are in $\Y,\ZZ$ respectively.
Hence $\X$ is closed under syzygies.

Now it remains to prove that $\X$ is closed under extensions.
Let $0 \to L \to M \to N \to 0$ be an exact sequence with $L,N\in\X$.
Then there are exact sequences $0 \to Z_1 \to L\oplus L' \to Y_1 \to 0$ and $0 \to Z_2 \to N\oplus N' \to Y_2 \to 0$ with $Y_i\in\Y$ and $Z_i\in\ZZ$ for $i=1,2$.
There are pushout and pullback diagrams:

{\footnotesize
$$
\begin{CD}
@. 0 @. 0 \\
@. @VVV @VVV \\
@. Z_1 @= Z_1 \\
@. @VVV @VVV \\
0 \to @. L\oplus L' @>>> M\oplus L' @>>> N @. \to 0 \\
@. @VVV @VVV @| \\
0 \to @. Y_1 @>>> A @>>> N @. \to 0 \\
@. @VVV @VVV \\
@. 0 @. 0
\end{CD}
\qquad\qquad
\begin{CD}
@. @. 0 @. 0 \\
@. @. @VVV @VVV \\
0 \to @. Y_1 @>>> B @>>> Z_2 @. \to 0 \\
@. @| @VVV @VVV \\
0 \to @. Y_1 @>>> A\oplus N' @>>> N\oplus N' @. \to 0 \\
@. @. @VVV @VVV \\
@. @. Y_2 @= Y_2 \\
@. @. @VVV @VVV \\
@. @. 0 @. 0
\end{CD}
$$
}

\noindent
By assumption the upper row in the right diagram splits, and we get an exact sequence $0 \to Y_1\oplus Z_2 \to A\oplus N' \to Y_2 \to 0$.
There are pushout and pullback diagrams:

{\footnotesize
$$
\begin{CD}
@. 0 @. 0 \\
@. @VVV @VVV \\
@. Z_2 @= Z_2 \\
@. @VVV @VVV \\
0 \to @. Y_1 \oplus Z_2 @>>> A\oplus N' @>>> Y_2 @. \to 0 \\
@. @VVV @VVV @| \\
0 \to @. Y_1 @>>> Y @>>> Y_2 @. \to 0 \\
@. @VVV @VVV \\
@. 0 @. 0
\end{CD}
\qquad\qquad
\begin{CD}
@. @. 0 @. 0 \\
@. @. @VVV @VVV \\
0 \to @. Z_1 @>>> Z @>>> Z_2 @. \to 0 \\
@. @| @VVV @VVV \\
0 \to @. Z_1 @>>> M\oplus L'\oplus N' @>>> A\oplus N' @. \to 0 \\
@. @. @VVV @VVV \\
@. @. Y @= Y \\
@. @. @VVV @VVV \\
@. @. 0 @. 0
\end{CD}
$$
}

\noindent
Since a resolving subcategory is closed under extensions, we have $Y\in\Y$ and $Z\in\ZZ$.
The middle column in the right diagram shows that $M\in\X$.
Thus $\X$ is closed under extensions.
\end{proof}

We denote by $\CM(R)$ the subcategory of $\mod R$ consisting of {\em maximal Cohen-Macaulay $R$-modules}, that is, $R$-modules $M$ such that $\depth M_\p\ge\height\p$ for all $\p\in\Spec R$.
The proposition below will be necessary to prove the main result of this section.

\begin{prop}\label{371520}
Let $R$ be a Gorenstein ring.
Let $\Y,\ZZ$ be resolving subcategories of $\mod R$ with $\Y\subseteq\PD(R)$ and $\ZZ\subseteq\CM(R)$.
Then one has $\Y=\res(\Y\cup\ZZ)\cap\PD(R)$ and $\ZZ=\res(\Y\cup\ZZ)\cap\CM(R)$.
\end{prop}

\begin{proof}
It is obvious that $\Y\subseteq\res(\Y\cup\ZZ)\cap\PD(R)$ and $\ZZ\subseteq\res(\Y\cup\ZZ)\cap\CM(R)$ hold.
Since $R$ is Gorenstein, $\Ext_R^1(M,P)=0$ for all $M\in\CM(R)$ and all $P\in\PD(R)$; see \cite[Theorems (4.13) and (4.20)]{AB}.
Hence $\Ext_R^1(Z,Y)=0$ for all $Y\in\Y$ and $Z\in\ZZ$.

Now take any $R$-module $M$ in $\res(\Y\cup\ZZ)$.
Lemma \ref{371446} implies that there exists an exact sequence
\begin{equation}\label{3714571}
0 \to Z \to M\oplus N \to Y \to 0
\end{equation}
with $Y\in\Y$ and $Z\in\ZZ$.

Since $Z$ is maximal Cohen-Macaulay, we have an exact sequence
\begin{equation}\label{3714572}
0 \to Z \to P \to \syz^{-1}Z \to 0,
\end{equation}
where $P$ is projective and $\syz^{-1}Z=\Hom_R(\syz(\Hom_R(Z,R)),R)$.
The pushout diagram of \eqref{3714571} and \eqref{3714572} gives rise to an exact sequence $0 \to M\oplus N \to Y' \to \syz^{-1}Z \to 0$ with $Y'\in\Y$.
Now suppose that $M$ has finite projective dimension.
Then, as $\syz^{-1}Z$ is maximal Cohen-Macaulay, we have $\Ext_R^1(\syz^{-1}Z,M)=0$.
Lemma \ref{371445}(1) implies that $M$ is isomorphic to a direct summand of $Y'$, and hence $M$ belongs to $\Y$.

Similarly, making the pullback diagram of \eqref{3714571} and $0 \to \syz Y \to Q \to Y \to 0$ with $Q$ projective and applying Lemma \ref{371445}(2), we deduce that $M$ is in $\ZZ$ if $M$ is maximal Cohen-Macaulay.
\end{proof}

Now we can prove our main result in this section.

\begin{thm}\label{lci}
Let $R$ be a locally complete intersection ring.
There exists a one-to-one correspondence
$$
\begin{CD}
\left\{
\begin{matrix}
\text{Resolving}\\
\text{subcategories}\\
\text{of $\mod R$}
\end{matrix}
\right\}
\begin{matrix}
@>{\Phi}>>\\
@<<{\Psi}<
\end{matrix}
\left\{
\begin{matrix}
\text{Resolving}\\
\text{subcategories}\\
\text{of $\PD(R)$}
\end{matrix}
\right\}
\times
\left\{
\begin{matrix}
\text{Resolving}\\
\text{subcategories}\\
\text{of $\CM(R)$}
\end{matrix}
\right\}.
\end{CD}
$$
Here $\Phi,\Psi$ are defined by $\Phi(\X)= (\X \cap \PD(R), \X \cap \CM(R))$ and $\Psi(\Y, \ZZ) = \res (\Y \cup \ZZ)$.
\end{thm}

\begin{proof}
Evidently, $\Phi$ and $\Psi$ are well-defined maps.
Proposition \ref{371520} guarantees that $\Phi\Psi=1$ holds.
It remains to show the equality $\Psi\Phi=1$.
Let $\X$ be a resolving subcategory of $\mod R$.
It is easy to see that $\Psi\Phi(\X)$ is contained in $\X$.
Take a module $X\in \X$ and set $n=\Rfd_RX<\infty$ (cf. \cite[Theorem 1.1]{AIL}).
Note that $\syz^nX$ is maximal Cohen-Macaulay by \cite[Proposition (2.3) and Theorem (2.8)]{CFF}.

By construction essentially given in \cite{AB,ABu}, we have an exact sequence
\begin{equation}\label{371608}
0 \to F \to C \to X \to 0
\end{equation}
of $R$-modules with $F\in\PD(R)$ and $C=\syz^{-n}\syz^nX\in\CM(R)$.
{For the convenience of the reader, we give a proof, since we
cannot find the precise reference for this fact.
There are exact sequences $0 \to \syz^nX \to P_{n-1} \to \cdots \to P_0 \to X \to 0$ and $\cdots \to Q_2 \to Q_1 \to Q_0 \to (\syz^nX)^\ast \to 0$, where all $P_i,Q_j$ are projective and $(-)^\ast=\Hom_R(-,R)$.
There exists a chain map from the $R$-dual complex of the first sequence to the second one that extends the identity map of $(\syz^nX)^\ast$, and taking its $R$-dual gives a commutatve diagram
$$
\begin{CD}
0 @>>> \syz^nX @>>> P_{n-1} @>>> \cdots @>>> P_0 @>>> X @>>> 0 \\
@. @| @A{f_0}AA @. @A{f_{n-1}}AA @A{f_n}AA \\
0 @>>> \syz^nX @>>> Q_0^\ast @>>> \cdots @>>> Q_{n-1}^\ast @>>> \syz^{-n}\syz^nX @>>> 0
\end{CD}
$$
with exact rows.
Adding free modules to the modules $Q_i$, we may assume that all $f_i$ are surjective.
Taking the kernels of $f_i$, we obtain a desired exact sequence.}

The pushout diagram of \eqref{371608} and an exact sequence $0 \to C \to G \to \syz^{-1}C \to 0$ with $G$ projective yields an exact sequence $0 \to X \to L \to \syz^{-1}C \to 0$, where $L\in\PD(R)$ and $\syz^{-1}C=\syz^{-n-1}(\syz^nX)\in\CM(R)$.
By Lemma \ref{2182108}(1), Proposition \ref{2182101} and \cite[Theorem 4.15]{radius}, the module $\syz^{-1}C$ belongs to $\res(\syz^nX)$.
Hence $\syz^{-1}C\in\X\cap\CM(R)$, and we see that $L\in\X\cap\PD(R)$ by the above exact sequence.
Consequently, $X$ belongs to $\res((\X\cap\PD(R))\cap(\X\cap\CM(R)))=\Psi\Phi(\X)$, and we conclude $\Psi\Phi=1$.
\end{proof}

Now let us recall the statement of Theorem \ref{hyper} and give a proof.

\begin{thmc}
\begin{enumerate}[\rm(1)]
\item
Let $R$ be a locally hypersurface ring.
There is a 1-1 correspondence
$$
\begin{CD}
\left\{
\begin{matrix}
\text{Resolving subcategories}\\
\text{of $\mod R$}
\end{matrix}
\right\}
\begin{matrix}
@>{\Phi}>>\\
@<<{\Psi}<
\end{matrix}
\left\{
\begin{matrix}
\text{Specialization closed}\\
\text{subsets of $\Sing R$}
\end{matrix}
\right\}
\times
\left\{
\begin{matrix}
\text{Grade consistent}\\
\text{functions on $\Spec R$}
\end{matrix}
\right\}.
\end{CD}
$$
One defines $\Phi,\Psi$ by $\Phi(\X)=(\IPD(\X),\phi(\X))$ with $\phi(\X)(\p)=\height\p-\min_{X\in\X}\{\depth X_\p\}$ and $\Psi(W,f)=\{\,M\in\IPD^{-1}(W)\mid\text{$\depth M_\p\ge\height\p-f(\p)$ for all $\p\in\Spec R$}\,\}$.
\item
Let $R=S/(\xx)$ where $S$ is a regular  ring and $\xx=x_1,\dots,x_c$ is a regular sequence on $S$.
Set $X=\P_S^{c-1}=\Proj S[y_1,\cdots,y_c]$ and let $Y$ be the zero subscheme of $\sum_{i=1}^cx_iy_i\in\Gamma(X,\OO_X(1))$.
Then one has a 1-1 correspondence
$$
\begin{CD}
\left\{
\begin{matrix}
\text{Resolving subcategories}\\
\text{of $\mod R$}
\end{matrix}
\right\}
\begin{matrix}
@>>>\\
@<<<
\end{matrix}
\left\{
\begin{matrix}
\text{Specialization closed}\\
\text{subsets of $\Sing Y$}
\end{matrix}
\right\}
\times
\left\{
\begin{matrix}
\text{Grade consistent}\\
\text{functions on $\Spec R$}
\end{matrix}
\right\}.
\end{CD}
$$
\end{enumerate}
\end{thmc}

\begin{proof}[{\bf Proof of Theorem \ref{hyper}}]
(1) By virtue of Theorems \ref{1.2} and \ref{lci}, we only need to show that the maps $\IPD$ and $\IPD^{-1}$ give mutually inverse bijections between the resolving subcategories contained in $\CM(R)$ and the specialization closed subsets of $\Sing R$. But this follows from \cite[Main Theorem]{stcm}, Lemma \ref{2182108}(1) and Proposition \ref{2182101}. 

(2) This assertion follows by combining Theorems \ref{1.2} and \ref{lci} with \cite[Corollary 4.16]{radius}, \cite[Proposition 6.2]{stcm} and \cite[Corollary 10.5 and Remark 10.7]{S}.
\end{proof}

\section*{Acknowledgments}
The authors are grateful to the referees for reading the paper carefully and giving valuable comments and useful suggestions.


\begin{thebibliography}{99}

\bibitem{AJS}
{\sc L. Alonso Tarr\'{i}o}; {\sc A. Jerem\'{i}as L\'{o}pez}; {\sc M. Saor\'{i}n}, Compactly generated $t$-structures on the derived category of a Noetherian ring, {\em J. Algebra} {\bf 324} (2010), no. 3, 313--346.

\bibitem{HPST}
{\sc L. Angeleri H\"{u}gel}; {\sc D. Posp\'{i}\v{s}il}; {\sc J. \v{S}\v{t}ov\'{i}\v{c}ek}; {\sc J. Trlifaj}, Tilting, cotilting, and spectra of commutative noetherian rings, \texttt{arXiv:1203.0907v1}.

\bibitem{Aus}
{\sc M. Auslander}, Modules over unramified regular local rings, Proc. Internat. Congr. Mathematicians (Stockholm, 1962), {\em Inst. Mittag-Leffler, Djursholm} (1963), 230--233.

\bibitem{AB}
{\sc M. Auslander}; {\sc M. Bridger}, Stable module theory, Memoirs of the American Mathematical Society, No. 94, {\em American Mathematical Society, Providence, R.I.}, 1969.

\bibitem{ABu}
{\sc M. Auslander}; {\sc R.-O. Buchweitz}, The homological theory of maximal Cohen-Macaulay approximations, Colloque en l'honneur de Pierre Samuel (Orsay, 1987), {\em M\'{e}m. Soc. Math. France (N.S.)} {\bf 38} (1989), 5--37.

\bibitem{AI}
{\sc L. L. Avramov}; {\sc S. B. Iyengar}, Constructing modules with prescribed cohomological support, {\em Illinois J. Math.} {\bf 51} (2007), 1--20.

\bibitem{AIL}
{\sc L. L. Avramov}; {\sc S. B. Iyengar}; {\sc J. Lipman}, Reflexivity and rigidity for complexes I: Commutative rings, {\em Algebra Number Theory} {\bf 4} (2010), no. 1, 47--86.

\bibitem{BCR}
{\sc D. J. Benson}; {\sc J. F. Carlson}; {\sc J. Rickard}, Thick subcategories of the stable module category, {\em Fund. Math.} {\bf 153} (1997), no. 1, 59--80.
\bibitem{BIK3}
{\sc D. J. Benson}; {\sc S. B. Iyengar}; {\sc H. Krause}, Stratifying triangulated categories, {\em J. Topol.} {\bf 4} (2011), no. 3, 641--666.

\bibitem{BIK}
{\sc D. J. Benson}; {\sc S. B. Iyengar}; {\sc H. Krause}, Stratifying modular representations of finite groups, {\em Ann. of Math. (2)} {\bf 174} (2011), no. 3, 1643--1684.

\bibitem{BIK2}
{\sc D. J. Benson}; {\sc S. B. Iyengar}; {\sc H. Krause}, A local-global principle for small triangulated categories, Preprint (2013), \texttt{arXiv:1305.1668}.

\bibitem{BM}
{\sc H. Bass}; {\sc M. P. Murthy}, Grothendieck groups and Picard groups of abelian group rings, {\em Ann. of Math. (2)} {\bf 86} (1967), 16--73.

\bibitem{BH}
{\sc W. Bruns; J. Herzog}, Cohen-Macaulay rings, revised edition, Cambridge Studies in Advanced Mathematics, 39, {\it Cambridge University Press, Cambridge}, 1998.

\bibitem{CFF}
{\sc L. W. Christensen}; {\sc H.-B. Foxby}; {\sc A. Frankild}, Restricted homological dimensions and Cohen-Macaulayness, {\em J. Algebra} {\bf 251} (2002), no. 1, 479--502.

\bibitem{rigid}
{\sc H. Dao}, Decent intersection and Tor-rigidity for modules over local hypersurfaces, {\em Trans. Amer. Math. Soc.} (to appear).

\bibitem{radius}
{\sc H. Dao; R. Takahashi}, The radius of a subcategory of modules, Preprint (2011), \texttt{arXiv:1111.2902v2}.

\bibitem{dim}
{\sc H. Dao; R. Takahashi}, The dimension of a subcategory of modules, Preprint (2012), \texttt{arXiv:1203.1955}.

\bibitem{DHS}
{\sc E. S. Devinatz}; {\sc M. J. Hopkins}; {\sc J. H. Smith}, Nilpotence and stable homotopy theory, I, {\em Ann. of Math. (2)} {\bf 128} (1988), no. 2, 207--241.

\bibitem{Ga}
{\sc P. Gabriel}, Des cat\'{e}gories ab\'{e}liennes, {\em Bull. Soc. Math. France} {\bf 90} (1962), 323--448.

\bibitem{G}
{\sc A. A. Gerko}, On homological dimensions, {\it Mat. Sb.} {\bf 192} (2001), no. 8, 79--94; translation in {\it Sb. Math.} {\bf 192} (2001), no. 7-8, 1165--1179.

\bibitem{H}
{\sc R. Hartshorne}, Algebraic geometry, Graduate Texts in Mathematics, No. 52, {\em Springer-Verlag, New York-Heidelberg,} 1977.

\bibitem{Hop}
{\sc M. J. Hopkins}, Global methods in homotopy theory, {\em Homotopy theory (Durham, 1985)}, 73--96, London Math. Soc. Lecture Note Ser., 117, {\em Cambridge Univ. Press, Cambridge}, 1987.

\bibitem{HS}
{\sc M. J. Hopkins}; {\sc J. H. Smith}, Nilpotence and stable homotopy theory, II, {\em Ann. of Math. (2)} {\bf 148} (1998), no. 1, 1--49.

\bibitem{K}
{\sc H. Krause}, Thick subcategories of modules over commutative Noetherian rings (with an appendix by Srikanth Iyengar), {\em Math. Ann.} {\bf 340} (2008), no. 4, 733--747.

\bibitem{Lich} 
{\sc S. Lichtenbaum}, {On the vanishing of Tor in regular local rings}, {\em Illinois J. Math.} {\bf 10} (1966), 220--226.

\bibitem{N}
{\sc A. Neeman}, The chromatic tower for $D(R)$, With an appendix by Marcel B\"{o}kstedt, {\em Topology} {\bf 31} (1992), no. 3, 519--532.

\bibitem{S}
{\sc G. Stevenson}, Subcategories of singularity categories via tensor actions, \texttt{arXiv:1105.4698}.

\bibitem{syz2}
{\sc R. Takahashi}, Syzygy modules with semidualizing or G-projective summands, {\it J. Algebra} {\bf 295} (2006), no. 1, 179--194.

\bibitem{wide}
{\sc R. Takahashi}, Classifying subcategories of modules over a commutative Noetherian ring, {\em J. Lond. Math. Soc. (2)} {\bf 78} (2008), no. 3, 767--782.

\bibitem{res}
{\sc R. Takahashi}, Modules in resolving subcategories which are free on the punctured spectrum, {\it Pacific J. Math.} {\bf 241} (2009), no. 2, 347--367.

\bibitem{stcm}
{\sc R. Takahashi}, Classifying thick subcategories of the stable category of Cohen-Macaulay modules, {\it Adv. Math.} {\bf 225} (2010), no. 4, 2076--2116.

\bibitem{arg}
{\sc R. Takahashi}, Contravariantly finite resolving subcategories over commutative rings, {\em Amer. J. Math.} {\bf 133} (2011), no. 2, 417--436.

\bibitem{T}
{\sc R. W. Thomason}, The classification of triangulated subcategories, {\em Compositio Math.} {\bf 105} (1997), no. 1, 1--27.
\end{thebibliography}
\end{document}